\documentclass[leqno]{amsart}

\usepackage{diagrams}

\usepackage{amssymb,
			enumitem, 
			caption, 
			array, 
			tikz, 
            xy 
}

\xyoption{all}

\DeclareCaptionFormat{myfigformat}{#1#2#3\hrulefill}
\DeclareCaptionFormat{mytabformat}{~ \\[10pt] #1#2#3\hrulefill}
\usepackage[
	bookmarksnumbered,
	pdftitle={quantum dilog AA}
	pdfauthor={Justin Allman}
	]
	{hyperref}
	
\usepackage{multirow}
	
\makeatletter
\newcommand{\leqnomode}{\tagsleft@true}
\newcommand{\reqnomode}{\tagsleft@false}
\makeatother

\newcommand{\curly}{\mathcal}
\newcommand{\G}{\boldsymbol{\mathrm{G}}}
\newcommand{\Rep}{\boldsymbol{\mathrm{Rep}}}
\DeclareMathOperator{\Hom}{Hom}
\DeclareMathOperator{\codim}{codim}
\newcommand{\GL}{\mathrm{GL}}
\newcommand{\U}{\mathrm{U}}
\newcommand{\kp}{\vdash}
\newcommand{\yy}{\mathbf{y}}
\newcommand{\cP}{\curly{P}}
\newcommand{\dsqp}{\vDash}
\newcommand{\cQ}{\curly{Q}^\bullet}

\newcommand{\R}{\mathbb{R}}
\newcommand{\Q}{\mathbb{Q}}
\newcommand{\Z}{\mathbb{Z}}
\newcommand{\C}{\mathbb{C}}
\newcommand{\E}{\mathbb{E}}
\newcommand{\N}{\mathbb{N}}
\newcommand{\A}{\mathbb{A}}

\newcommand{\union}{\cup}
\newcommand{\Union}{\bigcup}

\newcommand{\Dirsum}{\bigoplus}

\newcommand{\Tensor}{\bigotimes}

\newcommand{\til}{\widetilde}
\newcommand{\wh}{\widehat}
\newcommand{\iso}{\cong}
\newcommand{\homeo}{\approx}
\newcommand{\hmtpc}{\simeq}

\theoremstyle{plain}
\newtheorem{prop}{Proposition}[section]
\newtheorem{lem}[prop]{Lemma}

\newtheorem{cor}[prop]{Corollary}
\newtheorem{thm}[prop]{Theorem}
\newtheorem*{thm*}{Theorem}

\theoremstyle{definition}

\theoremstyle{remark}
\newtheorem{remark}[prop]{Remark}
\newtheorem{example}[prop]{Example}

\title[Interpolating factorizations for acyclic DT invariants]{Interpolating factorizations for acyclic {D}onaldson--{T}homas invariants}

\author[J.~Allman]{Justin Allman}
\address{Department of Mathematics \\ US Naval Academy \\ Annapolis, MD}
\email{allman@usna.edu}


\begin{document}

\begin{abstract}
We prove a family of factorization formulas for the combinatorial Donaldson--Thomas invariant for an acyclic quiver. A quantum dilogarithm identity due to Reineke, later interpreted by Rim\'anyi by counting codimensions of quiver loci, gives two extremal cases of our formulation in the Dynkin case. We establish our interpolating factorizations explicitly with a dimension counting argument by defining certain stratifications of the space of representations for the quiver and calculating Betti numbers in the corresponding equivariant cohomology algebras.
\end{abstract}

\maketitle


\section{Introduction}
\label{s:intro}

In a 2010 paper \cite{mr2010}, Reineke proved factorization formulas related to wall-crossing phenomena of Donaldson--Thomas (DT) type invariants associated to quivers. The invariants were described by Kontsevich--Soibelman \cite{mkys2014}. Given an acyclic quiver, Reineke associated a product to each \emph{discrete stability condition} (aka \emph{central charge}). The content of Reineke's theorem is that the products are ``invariants'' since they are {independent} from the choice of stability condition. Reineke's products were interpreted by Keller \cite{bk2011} as factorizations of a \emph{refined (aka combinatorial) DT invariant}, herein denoted $\E_Q$, where each factor is a \emph{quantum dilogarithm series} in the algebra $\Q(q^{1/2})[[z]]$. These series have a rich history in their own right; we refer the reader to the sources \cite{dz1988}, \cite{dz2007}, and \cite{lfrk1994} for some of their remarkable properties.

Keller's analysis expanded the discussion from acyclic quivers to the more general setting of quivers with potential, and furthermore initiated the study of so-called \emph{maximal green sequences} of quiver mutations. In particular, Keller described an algorithm which, to every maximal green sequence of length $r$, associates a factorization of the refined DT invariant with $r$ factors, each of which is a quantum dilogarithm series. We do not give a robust description of maximal green sequences and/or quiver mutations in this paper (instead we refer the reader to \cite{bk2010}), but as an example, consider the equioriented $A_3$ quiver $1 \leftarrow 2 \leftarrow 3$. The sequence given by mutating the quiver at vertices $2,3,2,1$ (in that order) is maximal green and corresponds to the following four-term product of quantum dilogarithms
\begin{equation}
	\label{eqn:mgs.ex}
\E_Q = \E(y_{\alpha_1})\E(y_{\alpha_3})\E(y_{\alpha_2+\alpha_3})\E(y_{\alpha_2})
\end{equation}
which is an element in the \emph{completed quantum algebra} of the quiver (see Sections \ref{s:prelims} and \ref{s:q.dilog} for precise definitions of the above).

One important application of Reineke's result is to quivers which are orientations of simply-laced Dynkin diagrams, the so-called \emph{Dynkin quivers}. In this case, Reineke's factorization result applied to the ``extreme'' stability conditions says (stated with details as our Theorem \ref{thm:reineke.rimanyi})
\begin{equation}
	\label{eqn:intro.dynkin.ex}
\E(y_{\alpha_1}) \cdots \E(y_{\alpha_n}) = \E_Q = \E(y_{\beta_1}) \cdots \E(y_{\beta_N})
\end{equation}
where the lefthand factorization is in terms of \emph{simple roots} $\alpha_1,\ldots,\alpha_n$ for the associated root system, and the righthand factorization is in terms of the corresponding \emph{positive roots} $\beta_1,\ldots,\beta_N$. In 2013, Rim\'anyi reinterpreted each side of the above identity through a dimension counting argument for quiver loci \cite{rr2013}, which are orbits in the space of quiver representations. These are important geometric objects with remarkable combinatorial properties, see e.g.~\cite{ab2008}, \cite{rr2014}, \cite{ja2014.ir}, and \cite{rk2017} for examples and current progress on this front. Rim\'anyi's method gives an explicit description, i.e.~not algorithmic, of the product on each side (which are the extremal examples of our factorizations) by computing Betti numbers in the equivariant cohomology algebras of the quiver loci. We call this technique the \emph{topological viewpoint}.

We observe that \eqref{eqn:mgs.ex} interpolates between the lefthand and righthand sides of \eqref{eqn:intro.dynkin.ex} in the sense that it contains factors for each simple root, but only some of the positive roots of $A_3$. It is natural to ask if the product \eqref{eqn:mgs.ex} can also be obtained \emph{explicitly} from a topological viewpoint. Indeed this paper's major accomplishment is to produce, by one general explicit dimension-counting method, a family of interpolating factorizations including members coming from both the implicit maximal green sequence algorithm and the ``extreme'' factorizations akin to \eqref{eqn:intro.dynkin.ex}. To do so, we require a generalization in two directions.

First, we extend the application of the topological viewpoint from quiver orbits to stratifications of the representation space in which each stratum is a union of orbits. We describe these strata in terms of rank and transversality conditions on elements in the space of quiver representations. We then calculate the Betti numbers of our strata and use a spectral sequence argument and properties of the quantum algebra associated to the quiver to obtain quantum dilogarithm identities. To do so, we must first introduce the concept of a \emph{Dynkin subquiver partition} of a quiver, and establish new necessary and sufficient admissibility criteria to properly order subsets of positive roots of the Dynkin diagrams. This is our Theorem \ref{thm:total.order.exists}. The root orderings provide explicit directions for how to multiply in the quantum algebra.

Second, with the extra freedom provided by taking unions of quiver orbits, we are able to recover the more general setting from \cite{mr2010} of acyclic quivers (not just Dynkin quivers). In the Dynkin case, this results in a family of factorization formulas for $\E_Q$ which interpolate between the two sides of Reineke's (and Rim\'anyi's) identity \eqref{eqn:intro.dynkin.ex}, and for which \eqref{eqn:mgs.ex} is an example of one of the intermediate formulas. Our general result concerning all acyclic quivers is Theorem \ref{thm:main}. 

The paper is organized as follows. In Section \ref{s:prelims} we give the necessary preliminary definitions, results, and notations needed throughout the paper. In particular, we construct the stratifications on the space of quiver representations which play a role throughout and prove admissibility conditions on stratifications. In Section \ref{s:q.dilog} we define the quantum dilogarithm series and relate it to relevant equivariant cohomology algebras. In Section \ref{s:dilogs.DT} we state our main theorem, and relate it to the context of Reineke's and Rim\'anyi's factorization formulas. In Section \ref{s:count.codim} we perform an important calculation in the quantum algebra of the quiver which produces the geometric data of the codimensions of our quiver strata. In Sections \ref{s:reduction.normal.forms} and \ref{s:KSS} we study the equivariant geometry and topology of our quiver strata, and describe our method for computing and relating their Betti numbers. Finally in Section \ref{s:main.thm.pf} we prove the main Theorem \ref{thm:main} and in Section \ref{s:FD} we describe several possible directions for future application and generalization of the present work.

\subsection*{Acknowledgements}

Throughout a portion of the period under which this work was completed, the author was supported by a grant from the Office of Naval Research (ONR) and Junior Naval Academy Research Council (NARC).


\section{Quiver preliminaries}
\label{s:prelims}

\subsection{Quivers and quiver representations}
	\label{ss:quivers.representations}

A quiver $Q = (Q_0,Q_1)$ is a finite set of vertices $Q_0$ and a finite set of directed edges $Q_1$ called \emph{arrows}. The direction of each arrow is encoded by associating to each $a\in Q_1$ its \emph{tail} and respectively \emph{head} vertices, denoted $ta\in Q_0$ and respectively $ha\in Q_0$. A quiver is \emph{acyclic} if it contains no oriented cycles (in particular, no loop arrows). In the rest of the paper, we only consider acyclic quivers; in fact, throughout the rest of the exposition, we fix an acyclic quiver $Q$ and order the vertices $Q_0 = \{1,2,\ldots,n\}$ so that the head of any arrow comes before its tail. This is always possible for acyclic quivers, see e.g.~\cite[1.5.2]{hdjw2017}. We will justify this choice of ordering for our purposes later, see Section \ref{ss:order.roots}.

A \emph{dimension vector} $\gamma = (\gamma(i))_{i\in Q_0}$ is a list of non-negative integers, one associated to each vertex. We let $D_Q = \N^{Q_0}$ denote the monoid of dimension vectors for $Q$. For $i\in Q_0$, let $e_i \in D_Q$ be the dimension vector with $e_i(i) = 1$ and $e_i(j) = 0$ for all other $j\neq i$. Thus we can identify $D_Q = \Dirsum_{i\in Q_0} \N\,e_i$.

Given a dimension vector $\gamma \in D_Q$, we can associate vector spaces $V_i = \C^{\gamma(i)}$ to each $i\in Q_0$, and consequently we form the vector space 
	\[ 
		\Rep_\gamma(Q) = \Dirsum_{a\in Q_1} \Hom\left(V_{ta},V_{ha}\right).
	\] 
Each element of $\Rep_\gamma(Q)$ is called a \emph{quiver representation} and amounts to a choice of linear mapping along each arrow and hence, after choosing bases, assigning an appropriately sized matrix to each $a\in Q_1$. The \emph{base-change group} $\G_\gamma = \prod_{i\in Q_0} \GL(V_i)$ acts on $\Rep_\gamma(Q)$ by changing bases in the head and tail of each arrow; i.e. via 
	\[ 
		(g_i)_{i\in Q_0} \cdot (x_a)_{a\in Q_1} = (g_{ha} x_a g_{ta}^{-1})_{a\in Q_1}
	\]
for $(g_i)_{i\in Q_0} \in \G_\gamma$ and $(x_a)_{a\in Q_1} \in \Rep_\gamma(Q)$.

Furthermore, for any quiver we define its $\N$-bilinear \emph{Euler form} $\chi:D_Q \times D_Q \to \Z$ by the formula
	\[
		\chi(\gamma_1,\gamma_2) 
			= \sum_{i\in Q_0} \gamma_1(i)\gamma_2(i) 
				- \sum_{a\in Q_1} \gamma_1(ta)\gamma_2(ha).
	\]
We will need the opposite anti-symmetrization of the Euler form given by
	\[
		\lambda(\gamma_1,\gamma_2) = \chi(\gamma_2,\gamma_1) - \chi(\gamma_1,\gamma_2).
	\]
Observe that $\lambda(e_i,e_j)$ is the number of arrows $i\to j$ minus the number of arrows $j \to i$; i.e.
	\begin{equation}
		\label{eqn:lambda.trees}
		\lambda(e_i,e_j) = 
		\left| \{a\in Q_1:ta = i,\,ha=j\} \right|
			- \left| \{a\in Q_1:ta = j,\,ha=i\} \right|	.
	\end{equation}
Hence our choice for ordering the vertices $Q_0 = \{1,\ldots,n\}$ implies by Equation \eqref{eqn:lambda.trees} that $\lambda(e_i,e_j) \leq 0$ whenever $i<j$. Because $\lambda$ counts arrows (and encodes their direction with a sign), given any subset of arrows $A\subseteq Q_1$, we can define the following restriction
	\[
	\lambda_A(e_i,e_j) = \left|\{a \in A : ta=i, ha=j\}\right| - \left|\{a\in A : ta=j, ha=i\}\right|
	\]
which we will need to satisfy a technical punctilio in Section \ref{ss:order.roots}. As in Equation \eqref{eqn:lambda.trees}, we agree to write $\lambda_{Q_1} = \lambda$.

\subsection{The quantum algebra of a quiver}
	\label{ss:quantum.alg}

Let $q^{1/2}$ be a variable and denote its square by $q$. The \emph{quantum algebra} $\A_Q$ of the quiver $Q$ is the $\Q(q^{1/2})$-algebra generated by symbols $\{y_\gamma:\gamma\in D_Q\}$ and satisfying the relations
	\begin{equation}
		\label{eqn:qalg.rel}
		y_{\gamma_1+\gamma_2} = -q^{-\frac{1}{2} \lambda(\gamma_1,\gamma_2)}y_{\gamma_1}y_{\gamma_2}.
	\end{equation}
The symbols $y_\gamma$ form a basis of $\A_Q$ as a vector space and the elements $\{y_{e_i}:i\in Q_0\}$ generate $\A_Q$ as an algebra. Notice that the relation \eqref{eqn:qalg.rel} implies the commutation relation
	\begin{equation}
		\label{eqn:qalg.comm}
		y_{\gamma_1}y_{\gamma_2} = q^{\lambda(\gamma_1,\gamma_2)} y_{\gamma_2}y_{\gamma_1}.
	\end{equation}
Furthermore, we let $\wh{\A}_Q$ denote the \emph{completed quantum algebra} in which formal power series in the symbols $y_\gamma$ are allowed, i.e.~the quotient of $\Q(q^{1/2})\langle\langle\{y_\gamma\}\rangle\rangle$ modulo relations given by \eqref{eqn:qalg.rel}. This is indeed the completion of $\A_Q$ with respect to the ideal generated by the elements $y_{e_i}$, $i\in Q_0$, and is sometimes in the literature called the \emph{formal quantum affine space}, see e.g.~\cite{bk2011}.

\subsection{Dynkin quivers}
	\label{ss:dynkin.quivers}

A \emph{Dynkin quiver} $Q$ is an orientation of a simply-laced Dynkin diagram; i.e.~of type $A$, $D$, or $E$. Fix a set of simple roots for the Lie algebra corresponding to the underlying Dynkin diagram. These are in bijection with the vertices of the quiver and we write $\alpha_i$ for the simple root associated to the vertex $i\in Q_0$. Furthermore, we can associate $\alpha_i$ with the dimension vector $e_i\in D_Q$, and this is done freely and without comment in the sequel. 

Let $\Phi$ denote the corresponding set of positive roots. If $|\Phi| = N$ write $\Phi = \{\beta_1,\ldots,\beta_N\}$ where, again, we postpone further discussion of this ordering until Section \ref{ss:order.roots}. Each positive root has a unique decomposition as a sum of simple roots which we will write as
	\[
		\beta_v = \sum_{i\in Q_0} d^i_v \alpha_i
	\] 
with $1\leq v \leq N$ and each $d^i_v \in \N$. Observe that this naturally identifies each positive root $\beta_v$ with a dimension vector in $D_Q$.

Given any dimension vector $\gamma \in D_Q$, a \emph{Kostant partition} of $\gamma$ is a list of non-negative integers $m = (m_v)_{1\leq v \leq N}$ such that
\[
	\gamma = \sum_{v=1}^N m_v \beta_v \in D_Q.
\]
In this case, we write $m\kp \gamma$. A now-classical result of Gabriel \cite{pg1972} implies that there are finitely many $\G_\gamma$-orbits in $\Rep_\gamma(Q)$, for all $\gamma$, exactly when $Q$ is a Dynkin quiver. In particular, when $Q$ is Dynkin, the orbits are in one-to-one correspondence with Kostant partitions of $\gamma$. Throughout the rest of the paper, we let $\Omega_m \subseteq \Rep_\gamma(Q)$ denote the $\G_\gamma$-orbit associated to the Kostant partition $m\kp\gamma$.

\begin{example}
	\label{ex:orbit}
Let $Q$ be the $A_3$ quiver $1 \leftarrow 2 \leftarrow 3$ with dimension vector $\gamma = (2,3,2)$. Let the six positive roots be ordered as follows
\begin{align*}
\beta_1 &= \alpha_3 					& \beta_2 &= \alpha_2+\alpha_3 	& \beta_3 & = \alpha_2 \\
\beta_4 &= \alpha_1+\alpha_2+\alpha_3 	& \beta_5 &= \alpha_1+\alpha_2	& \beta_6 & = \alpha_1
\end{align*}
and consider the Kostant partition $m_1 = m_3 = m_4 = m_5 = 1$ and $m_2 = m_6 = 0$. Since the positive roots in type $A$ correspond to intervals, this can be represented by the picture
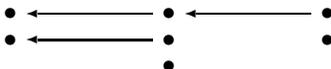
\[
	\begin{diagram}[height=1em]
	\bullet 	& \lTo 	& \bullet 	& \lTo 	& \bullet \\
	\bullet 	& \lTo 	& \bullet 	& 		& \bullet \\
				&		& \bullet	&		&			
	\end{diagram}
\]
which is called a \emph{lacing diagram}. The lacing diagrams were first used by Abeasis and Del Fra to parameterize type A quiver orbits \cite{saadf1980}; such diagrams do not exist in types D or E, but we present the diagram here to make the following description of the orbits more clear. Explicitly, from the picture it is straightforward to describe the orbit $\Omega_m \subseteq \Rep_\gamma(Q)$ geometrically. In particular, $\Omega_m$ consists of those quiver representations for which the linear map on the lefthand arrow $1 \leftarrow 2$ has maximum possible rank (i.e.~$2$), the linear map on the righthand arrow $2 \leftarrow 3$ has rank $1$, and such that the image of the righthand arrow is transverse to the kernel of the lefthand arrow. The upshot is that each Kostant partition $m \kp \gamma$ amounts to a choice of specific rank and transversality conditions on the maps along each arrow. This idea still passes to types D and E, even though the lacing diagrams do not.
\end{example}

For fixed $m\kp\gamma$, we choose a distinguished point $\nu_m \in \Omega_m$ which we call the \emph{normal form} of $\Omega_m$. The terminology comes from analogy with Smith Normal Form of a matrix (which is a distinguished point in the orbit of an $A_2$ quiver) or, for a non-Dynkin example, the Jordan Normal Form (which is a distinguished point in the orbit of a quiver with one vertex and one loop arrow), see respectively \cite{hdjw2017} Examples 1.4.2 and 1.4.3. Sometimes, to simplify or organize computations one can be quite specific about \emph{how} to choose a normal form, see e.g.~\cite[Section~2.6]{jarr2018}, but in the context of the present paper we only need to fix \emph{some} normal form; any point in $\Omega_m$ suffices for our arguments. In particular, we will later need the following result of Feher--Rimanyi.

\begin{prop}[Proposition~3.6, \cite{lfrr2002.duke}]
	\label{prop:FR}
Let $\G_{\Omega_m} \subseteq \G_\gamma$ denote the isotropy subgroup of $\Omega_m \subseteq \Rep_\gamma(Q)$, more precisely the isotropy (aka stabilizer) subgroup of the point $\nu_m \in \Omega_m$. Then, up to homotopy
\[
\G_{\Omega_m} \hmtpc \prod_{v=1}^N \U(m_v)
\]
where $\U(k)$ denotes the unitary group of $k\times k$ matrices.
\end{prop}

Observe that the stabilizer subgroup of any point in $\Omega_m$ is conjugate isomorphic to that for $\nu_m$; this is the sense in which our specific choice of the normal form is not important.

\subsection{Dynkin subquiver partitions}
	\label{ss:subquiver.parts}
	
We begin this subsection by reviewing several definitions. A \emph{subquiver} $Q'$ of $Q$ is a quiver with $Q'_0 \subseteq Q_0$ and $Q'_1 \subseteq Q_1$. A quiver $Q$ is \emph{connected} if its underlying non-oriented graph is connected and is \emph{nonempty} if $Q_0 \neq \emptyset$.

Let $Q^1,\ldots,Q^\ell$ be disjoint, nonempty, connected, Dynkin subquivers of $Q$ such that $Q_0 = \Union_{j=1}^\ell Q^j_0$. In particular, each vertex $i\in Q_0$ appears in exactly one of the subquiver vertex sets $Q^j_0$. On the other hand, notice that not every $a\in Q_1$ is required to appear in one of the arrow sets $Q^j_1$. We call the data of the quivers $Q^1,\ldots,Q^\ell$ a \emph{Dynkin subquiver partition of $Q$} and we write $\cQ = \{Q^1,\ldots,Q^\ell\} \dsqp Q$.

If $D^j := D_{Q^j}$ denotes the monoid of dimension vectors for $Q^j$, then there is a natural inclusion $D^j \subseteq D_Q$ by putting zero in the component for every vertex in $Q_0\setminus Q^j_0$. Analogously, for each subquiver $Q^j$ let $\Phi^j$ denote its set of positive roots, and observe that $\Phi^j$ is naturally a subset of $\Phi$. Let $r_j = |\Phi^j|$ and we introduce the following notation for the positive roots in $\Phi^j$:
\[
\Phi^j = \{\beta^j_1,\beta^j_2,\ldots,\beta^j_{r_j}\}
\]
where still we postpone discussion on our choice for ordering these roots until Section \ref{ss:order.roots}.

For $\gamma\in D_Q$ let $\gamma^j \in D^j$ denote the restriction of $\gamma$ to the vertices $Q^j_0$. We will call an ordered list $m = (m^1,\ldots,m^\ell)$ such that each $m^j \kp \gamma^j$ a \emph{Kostant series} of $\gamma$. Since this definition of Kostant series depends \emph{a priori} on the Dynkin subquiver partition $\cQ$, we say that such an $m$ is \emph{compatible (with $\cQ$)}. In abuse of notation, we also write $m \kp \gamma$, even though we understand that the symbology depends on the choice of $\cQ$. Recall that each Kostant partition $m^j\kp \gamma^j$ is associated to a quiver orbit $\Omega_{m^j} \subseteq \Rep_{\gamma^j}(Q^j)$, and we define the \emph{quiver strata associated to $m$} to be the subspace
\begin{equation}
\label{eqn:quiver.strata.defn}
\eta_m = \left\{ (x_a)_{a\in Q_1} \in \Rep_\gamma(Q) : (x_a)_{a\in Q^j_1} \in \Omega_{m^j} \text{ for all $j$} \right\} .
\end{equation}

\begin{remark}[On our terminology]
As a sequence of numbers, such an $m$ actually satisfies the condition to be a Kostant partition of $\gamma$ in the sense that $\gamma = \sum_{j=1}^\ell\sum_{k=1}^{r_j} m^j_k \beta^j_k$; although, we have only defined the words ``Kostant partition'' when $Q$ is itself Dynkin. However from the ``Kostant partition'' point of view, the components of elements in $\Omega_m$ corresponding to arrows not included in any $Q^j$ are zero, while in \eqref{eqn:quiver.strata.defn} the components along these arrows can assume arbitrary values from $\Hom(V_{ta},V_{ha})$. This justifies the use of the new term, i.e.~``Kostant series'', whenever this is our intention.
\end{remark}

\begin{example}
	\label{ex:subquiv.part}
Again suppose that $Q$ is the $A_3$ quiver: $1 \leftarrow 2 \leftarrow 3$. There are $4$ possible Dynkin subquiver partitions of $Q$ (two with $\ell = 2$). These are listed in Table \ref{tab:subquiv.example}.

\begin{table}
	\begin{tabular}{|c|c|}
	\hline
	$\ell$ 	& {$\cQ$} \\
	\hline
	$1$		& {$Q^1 = Q$} \\
	\hline
	$2$		& $Q^1 = \{1\}$, $Q^2 = \{2 \leftarrow 3\}$ or $Q^1 = \{1\leftarrow 2\}$, $Q^2 = \{3\}$ \\
	\hline
	$3$		& {$Q^1 = \{1\}$, $Q^2 = \{2\}$, $Q^3 = \{3\}$}\\
	\hline
	\end{tabular}
	\caption{Dynkin subquiver partitions for the quiver $1 \leftarrow 2 \leftarrow 3$.}
	\label{tab:subquiv.example}
\end{table}
Let $\gamma \in D_Q$. In the case $\ell = 1$ a compatible Kostant series $m\kp\gamma$ is exactly a Kostant partition and the quiver strata $\eta_m$ is exactly the quiver orbit $\Omega_m$. In the case $\ell = 3$, there is a unique compatible Kostant series with $m^1 = ((\gamma(1)))$, $m^2 = ((\gamma(2)))$, $m^3 = ((\gamma(3)))$. 

For the interpolating case $\ell = 2$, let us examine the choice $Q^1 = \{1\}$, $Q^2 = \{2 \leftarrow 3\}$ further. We have $\Phi^1 = \{\beta^1_1\}$ and $\Phi^2 = \{\beta^2_1,\beta^2_2,\beta^2_3\}$ with
\begin{align*}
\beta^1_1 &= \alpha_1 & \beta^2_1 &= \alpha_3 & \beta^2_2 &= \alpha_2 + \alpha_3 & \beta^2_3 &= \alpha_2.
\end{align*}
Fix the dimension vector $\gamma = (2,3,2)$ and consider the compatible Kostant series
\[
m = \left( (2) , (1,1,2) \right) \kp \gamma
\]
that is $m^1_1 = 2$, $m^2_1 = 1$, $m^2_2 = 1$, and $m^2_3 = 2$. The quiver stratum $\eta_m$ is then the set of quiver representations for which the linear mapping along the arrow $2 \leftarrow 3$ has rank $1$. Notice $\eta_m$ is not an orbit, but is the union of the $5$ orbits parameterized by the lacing diagrams below.
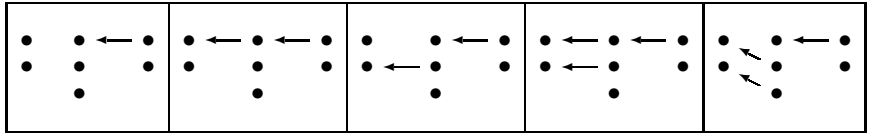
\begin{center}
\begin{tabular}{|c|c|c|c|c|}
\hline
&&&&\\
	{\begin{diagram}[height=1em,width=1em]
	\bullet 	&  & \bullet 	& \lTo	& \bullet \\
	\bullet 	&  & \bullet 	& 		& \bullet \\
			 	&  & \bullet 	& 		&
	\end{diagram}} 
	&
	{\begin{diagram}[height=1em,width=1em]
	\bullet 	& \lTo & \bullet 	& \lTo	& \bullet \\
	\bullet 	&  & \bullet 	& 		& \bullet \\
			 	&  & \bullet 	& 		&
	\end{diagram}}  
	& 
	{\begin{diagram}[height=1em,width=1em]
	\bullet 	&  & \bullet 	& \lTo	& \bullet \\
	\bullet 	& \lTo & \bullet 	& 		& \bullet \\
			 	&  & \bullet 	& 		&
	\end{diagram}}  
	& 
	{\begin{diagram}[height=1em,width=1em]
	\bullet 	& \lTo & \bullet 	& \lTo	& \bullet \\
	\bullet 	& \lTo & \bullet 	& 		& \bullet \\
			 	&  & \bullet 	& 		&
	\end{diagram}}  
	& 
	{\begin{diagram}[height=0.5em,width=1em]
	\bullet 	&  		& \bullet 	& \lTo	& \bullet \\
				&\luTo 	& 			& 		& \\
	\bullet 	&  		& \bullet 	& 		& \bullet \\
				&\luTo 	& 			& 		& \\
			 	&  		& \bullet 	& 		&
	\end{diagram}} \\
&&&&\\
\hline
\end{tabular}
\end{center}
\end{example}

In fact, we can make a precise statement about the dimension of quiver strata in terms of the orbits $\Omega_{m^j}$ in the representation spaces of each subquiver $\Rep_{\gamma^j}(Q^j)$. Since we are imposing no extra rank or transversality conditions on the mappings along arrows $a\in Q_1\setminus(\Union_{1\leq j \leq \ell}Q^j_1)$, the next proposition follows immediately from the definitions, cf.~\cite[Proposition~7.2]{jarr2018}.

\begin{prop}
	\label{prop:codim.eta}
For a Kostant series $m\kp\gamma$ compatible with $\cQ$, we have
\begin{equation}
	\label{eqn:codim.strata}
\pushQED{\qed}
\codim_\C \left(\eta_m;\Rep_\gamma(Q)\right) = \sum_{j=1}^\ell \codim_\C \left(\Omega_{m^j};\Rep_{\gamma^j}(Q^j)\right).
\qedhere
\popQED
\end{equation}
\end{prop}

\subsection{Ordering Roots}
	\label{ss:order.roots}

 As in \cite{mr2001,mr2003,rr2013,jarr2018}, there exists a total ordering (not unique) for the positive roots $\beta_u^j \in \Phi^j$, for $1\leq u \leq r_j$
\begin{equation}
\label{eqn:pos.roots.order.within}
\beta^j_1 \prec \cdots \prec \beta^j_{r_j}
\end{equation}
satisfying the condition \[u<v \implies \lambda_{Q^j_1}(\beta^j_u,\beta^j_v) \geq 0.\] This ordering was originally described in terms of homological properties of indecomposable quiver representations, see e.g.~\cite[Section~6.2]{mr2010} and \cite[Section~4]{rr2013}; the equivalence with the conditions on $\lambda$ is established in \cite[Lemma~5.1]{jarr2018}.

Now let $r = \sum_{j=1}^\ell r_j$ and $\Phi(Q,\cQ) = \Union_{j=1}^\ell \Phi^j$, so $r = |\Phi(Q,\cQ)|$. In the remainder of the subsection, we show that for certain choices of $\cQ$, we can splice together the above orderings into a total ordering on the roots $\phi_u\in\Phi(Q,\cQ)$
\begin{equation}
\label{eqn:pos.roots.order.total}
\phi_1 \prec \cdots \prec \phi_r
\end{equation}
such that, after writing $\phi_u = \beta^j_k$ and $\phi_v = \beta^{j'}_{k'}$ for appropriate $j$, $j'$, $k$ and $k'$, the following conditions are satisfied:
\begin{subequations}
\label{eqn:order.rules}
	\begin{gather}
			\text{$j=j'$ and $u<v$}  \implies \lambda_{Q^j_1}(\phi_u,\phi_v) \geq 0 \text{~and~}
				\lambda_{Q_1 \setminus Q^j_1}(\phi_u,\phi_v) \leq 0; 		
					\label{eqn:order.rule.within.technical}
			\\
			\text{$j\neq j'$ and $u<v$}  \implies \lambda(\phi_u,\phi_v) \leq 0. 
					\label{eqn:order.rule.without}
	\end{gather}
\end{subequations}
The first condition ensures that the ordering rules within subquivers from \eqref{eqn:pos.roots.order.within} are preserved. Next, we describe the necessary hypotheses on the subquiver partition $\cQ$ for such an order to exist.

Given $\cQ \dsqp Q$ we define a new quiver, denoted $P(Q,\cQ)$, to be the quiver obtained from $Q$ by contracting each subquiver $Q^j \in \cQ$ to a single vertex. The resulting vertices will be labeled $Q^1,\ldots, Q^\ell$ in the quiver $P(Q,\cQ)$. Observe that the arrows $P(Q,\cQ)_1$ are identified with the set $Q_1\setminus \Union_{1\leq j \leq \ell} Q^j_1$. If $P(Q,\cQ)$ is acyclic,  then we will say that $\cQ$ is \emph{admissible}. In particular, observe that if $P(Q,\cQ)$ has no loops (aka 1-cycles), then for each $j$ we have $\lambda_{Q^j} = \lambda$ and the ordering rule \eqref{eqn:order.rule.within.technical} simplifies to
\[
\text{$j=j'$ and $u<v$}  \implies \lambda(\phi_u,\phi_v) \geq 0.
\]
Hence, for admissible $\cQ \dsqp Q$ we have the ordering rules for roots $\phi_u = \beta^j_k$ and $\phi_v = \beta^{j'}_{k'}$ as follows
\begin{equation}
	\label{eqn:order.rules.admissible}
	\begin{split}
	\text{$j=j'$ and $u<v$} & \implies \lambda(\phi_u,\phi_v) \geq 0; \\
	\text{$j \neq j'$ and $u<v$} & \implies \lambda(\phi_u,\phi_v) \leq 0.	
	\end{split}	
\end{equation}
If the subquivers $Q^j \in \cQ$ are labeled such that for every arrow $a\in P(Q,\cQ)_1$ we have $ha = Q^j$ and $ta=Q^{j'}$ with $j<j'$, then $\cQ$ is called \emph{ordered}. Every admissible $\cQ$ can be ordered since, in this case, $P(Q,\cQ)$ is acyclic (we have already used that acyclic quivers admit such a ``head before tail'' ordering on vertices, see Section \ref{ss:quivers.representations}).

\begin{thm}
	\label{thm:total.order.exists}
A total ordering \eqref{eqn:pos.roots.order.total} satisfying \eqref{eqn:order.rules.admissible} exists (but is not unique) if and only if $\cQ \dsqp Q$ is admissible. In particular, if $\cQ$ is ordered then an allowed total order on $\Phi(Q,\cQ)$ is
\begin{equation}
	\label{eqn:our.order.exists}
\Phi^1 \prec \cdots \prec \Phi^\ell,
\end{equation}
where each $\Phi^j$ is ordered according to the first rule in \eqref{eqn:order.rules.admissible}.
\end{thm}

\begin{proof}
Suppose that $\cQ$ is admissible and ordered. Since every arrow $a\in P(Q,\cQ)_1$ satisfies $j<j'$ where $ha = Q^j$, $ta = Q^{j'}$, we have, by repeated application of \eqref{eqn:lambda.trees}, that $\lambda(\beta^j_u,\beta^{j'}_v) \leq 0$ for every $1\leq u \leq r_j$ and $1\leq v \leq r_{j'}$. When $\cQ$ is ordered, $\beta^j_u$ will appear before $\beta^{j'}_v$ and we have proven \eqref{eqn:our.order.exists} is an allowed total ordering.

Conversely suppose that $P(Q,\cQ)$ has a $k$-cycle. We will treat the cases $k=1$, $k=2$, and $k\geq 3$ separately. For $k=1$, without loss of generality we assume that there exists a loop at the vertex $Q^1\in P(Q,\cQ)_0$. This loop comes from an arrow $a\in Q_1 \setminus \Union_{j} Q^j$. There must be a subquiver of type $A_p$ in $Q^1$ having endpoint vertices $ta$ and $ha$ as depicted below:
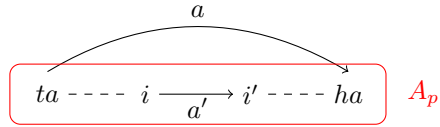
\begin{center}
	\begin{tikzpicture}
	\node (ta) at (0,0) {$ta$};
	\node (ha) at (4,0) {$ha$};
	\node (alabel) at (2,1.1) {$a$};
	\node (i) at (1.3,0) {$i$};
	\node (i') at (2.7,0) {$i'$};
	\node (a'label) at (2,-.18) {$a'$};
	\node [red] at (5,0) {$A_p$};
	
	\draw [->,domain=0:4] plot (\x, {.3+.6*\x-0.15*\x*\x)});
	\draw [dashed] (ta) -- (i);
	\draw [->] (i) -- (i');
	\draw [dashed] (i') -- (ha);
	
	\draw [rounded corners,red] (-.5,0) -- (-.5,.4) -- (4.5,.4) -- (4.5,-.4) -- (-.5,-.4) -- (-.5,0);
	\end{tikzpicture}
\end{center}
where, as in our drawing, at least some arrow $a'\in (A_p)_1$ points to the right and $p>1$ or else $Q$ is not acyclic. Let $\beta$ denote the root of $A_p$ corresponding to the interval $[ta, i]$ and let $\beta'$ denote the root of $A_p$ corresponding to the interval $[i', ha]$. Both $\beta$ and $\beta'$ must also be roots of $Q^1$. Now, $\lambda_{Q^1}(\beta,\beta')>0$ and $\lambda_{Q_1\setminus Q^1_1} (\beta,\beta')>0$ and hence $\beta$ and $\beta'$ must violate the conditions of \eqref{eqn:order.rule.within.technical}.

If $k=2$, then without loss of generality we assume a 2-cycle between the subquivers $Q^1$ and $Q^2$. We call the arrows forming the 2-cycle $a$ and $b$, and we must have the following picture as a subgraph of $Q$
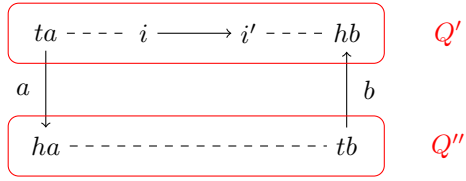
\begin{center}
	\begin{tikzpicture}
	\node (ta) at (0,1.5) {$ta$};
	\node (ha) at (0,0) {$ha$};
	\node (tb) at (4,0) {$tb$};
	\node (hb) at (4,1.5) {$hb$};
	\node (i) at (1.3,1.5) {$i$};
	\node (i') at (2.7,1.5) {$i'$};
	\node (alab) at (-.3,.75) {$a$};
	\node (blab) at (4.3,.75) {$b$};
	
	\draw [->] (ta) -- (ha);
	\draw [->] (tb) -- (hb);
	\draw [dashed] (ta) -- (i);
	\draw [dashed] (i') -- (hb);
	\draw [dashed] (ha) -- (tb);
	\draw [->] (i) -- (i');
	
	\draw[rounded corners,red] (-.5,0) -- (-.5,.4) -- (4.5,.4) -- (4.5,-.4) -- (-.5,-.4) -- (-.5,0);
	\draw[rounded corners,red] (-.5,1.5) -- (-.5,1.9) -- (4.5,1.9) -- (4.5,1.1) -- (-.5,1.1) -- (-.5,1.5);
	
	\node [red] at (5.35,1.5) {$Q'$};
	\node [red] at (5.35,0) {$Q''$};
\end{tikzpicture}
\end{center} 
where $Q'$ is a type $A$ subquiver of $Q^1$ with endpoints $ta$ and $hb$, while $Q''$ is a type $A$ subquiver of $Q^2$ with endpoints $ha$ and $tb$. Moreover, we have assumed without loss of generality that not both of $Q'$ and $Q''$ are $A_1$ and that there must be at least one rightward arrow in $Q'_1$ (or one leftward arrow in $Q''_1$), or else $Q$ is not acyclic. Now let $\beta'_1$, $\beta'_2$, and $\beta''$ denote the roots which correspond respectively to the intervals $[ta,i]$, $[i',hb]$, and $[ha,tb]$. To satisfy \eqref{eqn:order.rule.within.technical}, we need $\beta'_1 \prec \beta'_2$. On the other hand, \eqref{eqn:order.rule.without} implies $\beta''\prec \beta'_1$ and $\beta'_2\prec \beta''$ since $\lambda(\beta'', \beta'_1)$ and $\lambda(\beta'_2, \beta'')$ are both negative. But this gives a contradiction since trichotomy then implies $\beta'_2 \prec \beta'_1$.

In the case $k\geq 3$, we can find a contradiction to trichotomy in an ordering by comparing the longest roots in each of the subquivers making up the existing $k$-cycle.
\end{proof}

\begin{example}
	\label{ex:ordering}
Again, consider the quiver $1\leftarrow 2 \leftarrow 3$ with Dynkin subquiver partition $\cQ = \{ 1, 2\leftarrow 3\}$ as in Example \ref{ex:subquiv.part}. Then $P(Q,\cQ)$ is the acyclic quiver $Q^1 \leftarrow Q^2$. Further, observe that the ordering implied in Example \ref{ex:subquiv.part} verifies Theorem \ref{thm:total.order.exists}. In particular, setting
\begin{align*}
\phi_1 = \beta^1_1 &= \alpha_1 & \phi_2 = \beta^2_1 &= \alpha_3 & \phi_3 = \beta^2_2 &= \alpha_2 + \alpha_3 & \phi_4 = \beta^2_3 &= \alpha_2
\end{align*}
gives an allowed total ordering $\phi_1 \prec \phi_2 \prec \phi_3 \prec \phi_4$.
\end{example}

\begin{example}
	\label{ex:order.trees}
Observe more generally that if $Q$ is an orientation of a tree, then every Dynkin subquiver partition $\cQ$ is automatically admissible. In particular, this applies to every Dynkin quiver.
\end{example}

\begin{example}
	\label{ex:order.counterexamples}
Consider now the acyclic orientation of $Q=\til{A}_2$:
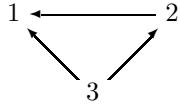
\[
	\begin{diagram}[height=1.5em,width=1.5em]
	1 	&		&\lTo	&		& 2	\\
		&\luTo	&		&\ruTo	& 	\\
		&		& 3		&		&
	\end{diagram}
\]
Several different Dynkin subquiver partitions of $Q$ illustrate the technicalities in our definitions and proof of Theorem \ref{thm:total.order.exists}. First, the subquiver partition $\cQ$ with $Q^1 = \{1\}$ and $Q^2 = \{2 \leftarrow 3\}$ is both admissible and ordered. Indeed, in this case $P(Q,\cQ)$ is the Kronecker quiver $Q^1 \overleftarrow{\leftarrow} Q^2$ which is acyclic and has vertices ordered with ``head before tail'' for each arrow. Hence, an allowed total ordering on the roots $\Phi(Q,\cQ)$ is given by
\begin{align*}
\phi_1 &= \alpha_1 & \phi_2 &= \alpha_3 & \phi_3 &= \alpha_2 + \alpha_3 & \phi_4 &=\alpha_2.
\end{align*}

On the other hand, the Dynkin subquiver partition $Q^1 = \{1 \leftarrow 3 \}$, $Q^2 = \{2\}$ is not admissible since the resulting $P(Q,\cQ)$ quiver is $Q^1 \overrightarrow{\leftarrow} Q^2$ which contains a 2-cycle. Moreover, as in the proof above we can check that the roots $\Phi(Q,\cQ)$ can not be totally ordered according to our rules. Indeed, to satisfy  \eqref{eqn:order.rule.within.technical} we must take
\[
\Phi^1 = \{ \alpha_3 \prec \alpha_1 + \alpha_3 \prec \alpha_1\}
	\quad\text{and}\quad
\Phi^2 = \{\alpha_2\}. 
\]
But then \eqref{eqn:order.rule.without} implies we must have $\alpha_1 \prec \alpha_2$ and $\alpha_2 \prec \alpha_3$, but this contradicts that $\alpha_3 \prec \alpha_1$ in our ordering of $\Phi^1$ above.

Finally, consider the Dynkin subquiver partition $Q^1 = \{1 \leftarrow 2 \leftarrow 3\}$. This is not admissible since $P(Q,\cQ)$ is the so-called Jordan quiver with one vertex $Q^1$ and one loop arrow (aka 1-cycle). Then \eqref{eqn:order.rule.within.technical} cannot be satisfied, say for $\beta' = \alpha_2+\alpha_3$ and $\beta'' = \alpha_1$ since 
\[
\lambda_{Q^1_1}(\beta',\beta'') = 1 > 0, \text{~but also~} \lambda_{Q_1\setminus Q^1_1}(\beta',\beta'') = 1 >0.
\]
We have illustrated how loops and 2-cycles in $P(Q,\cQ)$ violate our desired total orders; however, we postpone a discussion on \emph{why} we want to exclude non-admissible Dynkin subquiver partitions $\cQ$ until Section \ref{s:dilogs.DT} (in particular, see Example \ref{ex:Kronecker.DT}).
\end{example}


\section{Quantum dilogarithms and {P}oincar\'e series}
\label{s:q.dilog}

Let $z$ be an indeterminate. The element $\E(z) \in \Q(q^{1/2})[[z]]$ defined by the formula
\begin{equation}
\label{eqn:E.defn}
\E(z)= 1 + \sum_{k=1}^\infty \frac{(-z)^k q^{k^2/2}}{\prod_{j=1}^k (1-q^j)}
\end{equation}
is called the \emph{quantum dilogarithm series}. Below, we describe its connection to Poincar\'e series of equivariant cohomology algebras.

Given any $\N$-graded $\R$-algebra $\curly{A} = \Dirsum_{j\in \N} \curly{A}_j$, for which every graded piece $\curly{A}_j$ is a finite-dimensional $\R$-vector space, the \emph{Poincar\'e series} of $\curly{A}$ in the variable $q^{1/2}$ is
\[
\cP[\curly{A}] = \sum_{j\in\N} q^{j/2} \dim_\R(A_j).
\]

In the rest of the paper, the equivariant cohomology algebra $H^*_{\GL(k,\C)}(\C^k)$, where $\GL(k,\C)$ acts in the standard way on $\C^k$, will play an important role. Since $\C^k$ is equivariantly contractible, this algebra is isomorphic to $H^*_{\GL(k,\C)}(\text{point})$ which is by definition $H^*(B\GL(k,\C))$. Here, and throughout the sequel, $B$ stands for the \emph{Borel construction} in equivariant cohomology. Since $\GL(k,\C)$ is homotopy equivalent to its maximal compact subgroup $\U(k)$, we further have an isomorphism $H^*(B\GL(k,\C)) \cong H^*(B\U(k))$. These identifications will be used freely and without comment in the sequel.

Set $\cP_{k} = \cP[H^*(B\GL(k,\C))]$. The algebra $H^*(B\GL(k,\C))$ is a polynomial ring in the Chern classes $c_1$, $\ldots$, $c_k$ of $\GL(k,\C)$ with $\deg(c_i)=2i$, all of the odd cohomology groups vanish and we have \[\cP_k = \sum_{j\geq 0} q^j\,\dim_\R(H^{2j}(B\GL(k,\C))),\] from whence we obtain that $\curly P_0 = 1$ and $\cP_k = \prod_{j=1}^k (1-q^j)^{-1}$ for $k>0$. Thus the quantum dilogarithm series can be written as
\begin{equation}
\label{eqn:E.defn.poincare}
\E(z) = \sum_{k\geq 0} (-z)^k q^{k^2/2}\, \cP_{k}.
\end{equation}
We remark that often in the literature, see for example \cite{bk2011,bk2013.fpsac}, the quantum dilogarithm series is instead defined to be
\begin{equation}\label{eqn:KE.defn}
1 + \sum_{k\geq 1} \frac{z^k q^{k^2/2}}{(q^k-1)(q^k-q)\cdots(q^k-q^{k-1})}.
\end{equation}
Note that in the series \eqref{eqn:KE.defn}, the denominators count elements of the group $\GL(k,\mathbb{F}_q)$. The two formulations \eqref{eqn:E.defn} and \eqref{eqn:KE.defn} are images of each other under the involution $q^{1/2} \mapsto -q^{-1/2}$. However, it is more convenient for our purposes to count generators in $H^*(B\GL(k,\C)) \iso H^*(B\U(k))$ and therefore we choose to work with the formulation given by Equation \eqref{eqn:E.defn} and \eqref{eqn:E.defn.poincare}. This is consistent with the conventions in \cite{jarr2018} and \cite{rr2013}, where the techniques are more comparable to the present paper.


\section{Dilogarithm identities, Donaldson--Thomas invariants, and the statement of the main theorem}
\label{s:dilogs.DT}
Reineke proved a quantum dilogarithm factorization for acyclic quivers in \cite{mr2010}. Keller defined the \emph{combinatorial (aka refined) Donaldson--Thomas (DT) invariant of $Q$}, denoted $\E_Q$, to be the common value of the resulting product \cite{bk2011,bk2013.fpsac}. The application of Reineke's work to Dynkin quivers, see \cite[Section~6.2]{mr2010}, was interpreted as a dimension count for Dynkin quiver orbits by Rim{\'a}nyi \cite[Theorem~6.1]{rr2013}. We state this result below.

\begin{thm}[\cite{mr2010,rr2013}]
	\label{thm:reineke.rimanyi}
Let $Q = (Q_0,Q_1)$ be a Dynkin quiver with $n$ vertices. Let $\alpha_1, \ldots, \alpha_n$ denote its simple roots (in bijection with the set of vertices $Q_0$), ordered such that for every arrow $j\to i$ in $Q_1$, we have $\alpha_i \prec \alpha_j$ (i.e.~head before tail as in Section \ref{ss:dynkin.quivers}). Let $\beta_1,\ldots,\beta_N$ denote the positive roots ordered as in  \eqref{eqn:pos.roots.order.within}, i.e.~with $\ell=j=1$. Then in the completed quantum algebra $\wh\A_Q$ we have
\begin{equation}
\label{eqn:dilog.ident.extremes}
\pushQED{\qed}
\E(y_{\alpha_1})\E(y_{\alpha_2}) \cdots \E(y_{\alpha_n}) = \E(y_{\beta_1})\E(y_{\beta_2}) \cdots \E(y_{\beta_N}).\qedhere
\popQED
\end{equation}
\end{thm}

Even in the non-Dynkin (but still acyclic) case, applying Reineke's result to Keller's definition gives
\begin{equation}
	\label{eqn:trivial.factorization}
	\E_Q = \E(y_{e_1})\E(y_{e_2}) \cdots \E(y_{e_n})
\end{equation}
which is the lefthand side of \eqref{eqn:dilog.ident.extremes} in the Dynkin case. We will call \eqref{eqn:trivial.factorization} the \emph{trivial factorization of $\E_Q$}. We now remark on how Theorem \ref{thm:reineke.rimanyi} fits into our present point of view.

Observe that for $\ell=n$, and hence subquivers $Q^i = (\{i\},\emptyset)$, the ``head comes before tail'' ordering on simple roots on the lefthand side of \eqref{eqn:dilog.ident.extremes} is equivalent to \eqref{eqn:pos.roots.order.total}. In particular, we only need to use the second ordering rule from \eqref{eqn:order.rules.admissible}. This applies also to Equation \eqref{eqn:trivial.factorization}.

On the other hand, with $\ell=1$, and hence $Q^1 = Q$, we see that the ordering on the righthand side of \eqref{eqn:dilog.ident.extremes} is also equivalent to \eqref{eqn:pos.roots.order.total}, where this time we need only use the first rule from \eqref{eqn:order.rules.admissible}. 

Hence, the identity \eqref{eqn:dilog.ident.extremes} relates the extremal choices for Dynkin subquiver partitions $\cQ$ of $Q$. This motivates the following factorization formula for the combinatorial DT invariant, which interpolates between these extremes. 

\begin{thm}
\label{thm:main}
Let $Q$ be an acyclic quiver, and let $\cQ\dsqp Q$ be an admissible Dynkin subquiver partition of $Q$. Suppose that the roots $\{\phi_1\prec \cdots \prec\phi_r\} = \Phi(Q,\cQ)$ are ordered according to the rules \eqref{eqn:order.rules.admissible}. We have that
\begin{equation}
\label{eqn:dilog.ident.interpolated}
\E_Q = \E(y_{\phi_1}) \E(y_{\phi_2}) \cdots \E(y_{\phi_r}).
\end{equation}
\end{thm}

\noindent This is the main result of our paper; we prove it in Section \ref{s:main.thm.pf}.

\begin{example}
	\label{ex:Kronecker.DT}
Consider the Kronecker quiver $Q$: $1 \overleftarrow{\leftarrow} 2$. If we take the admissible, ordered, Dynkin subquiver partition $Q^1 = \{1\}$ and $Q^2 = \{2\}$ we obtain the trivial factorization
\[
\E_Q = \E(y_{e_1})\E(y_{e_2}).
\]
There are no other admissible $\cQ$ for this quiver. However, consider the non-admissible Dynkin subquiver partition with $Q^1 = \{1 \leftarrow 2\}$ so that $P(Q,\cQ)$ becomes the Jordan quiver with the single vertex $Q^1$ and one loop arrow (the \emph{other} arrow not in $Q^1_1$). We cannot satisfy both conditions in \eqref{eqn:order.rule.within.technical}, note that \eqref{eqn:order.rule.without} is trivially satisfied, but we can set
\begin{align*}
\phi_1 & = e_2 & \phi_2 & = e_1 + e_2 & \phi_3 & = e_1
\end{align*}
to satisfy the first condition of \eqref{eqn:order.rule.within.technical} \emph{inside} $Q^1$. There is a known factorization of $\E_Q$ beginning with $\E(y_{e_2})$, see e.g.~\cite[(1.6)]{bk2011}, but it has infinitely many terms! One complication is that although $e_1 + e_2$ is a root of $Q^1 \iso A_2$, it is not a root of $Q$ since $\chi(e_1+e_2,e_1+e_2) = 0$ (whereas $\chi(\beta,\beta) = 1$ for roots). Notice that the above order $\phi_1 \prec \phi_2 \prec \phi_3$ does naively satisfy \eqref{eqn:order.rules.admissible}, so the notion of \emph{admissibility} and in particular the rules \eqref{eqn:order.rule.within.technical} are necessary to ensure that the corresponding product of quantum dilogarithms is actually $\E_Q$, and not just an arbitrary element of $\wh{\A}_Q$.
\end{example}

\begin{example}
	\label{ex:equi.A3.DT}
Let $Q$ again be the equioriented $A_3$ quiver: $1\leftarrow 2 \leftarrow 3$. Referring to Table \ref{tab:subquiv.example} and ordering roots according to Theorem \ref{thm:total.order.exists} (say, as in Examples \ref{ex:orbit} and \ref{ex:ordering}), we see that Theorem \ref{thm:main} gives the four factorizations
	\begin{align*}
	\E_Q 	& = \E(y_{\alpha_1})\E(y_{\alpha_2})\E(y_{\alpha_3}) \\
			& = \E(y_{\alpha_1})\E(y_{\alpha_3})\E(y_{\alpha_2 + \alpha_3})\E(y_{\alpha_2}) \\
			& = \E(y_{\alpha_2})\E(y_{\alpha_1+\alpha_2})\E(y_{\alpha_1})\E(y_{\alpha_3}) \\
			& = \E(y_{\alpha_3})\E(y_{\alpha_2+\alpha_3})\E(y_{\alpha_2})
				\E(y_{\alpha_1+\alpha_2+\alpha_3})\E(y_{\alpha_1+\alpha_2})\E(y_{\alpha_1}).
	\end{align*}
The first line and last line are the two sides of Theorem \ref{thm:reineke.rimanyi}. We comment that the second and third lines can be obtained from the first through application of the quantum Pentagon Identity (see e.g.~\cite[Theorem~1.2]{bk2011}), but in the sequel, we establish these identities explicitly through geometric and topological methods.
\end{example}

\begin{remark}
Recall that the quantum dilogarithm $\E(z)$ is a power series whose coefficients are (up to a power of $q$) the rational functions $\cP_k = \prod_{1\leq j \leq k} (1-q^j)^{-1}$. Hence, the content of Theorem \ref{thm:main} comprises infinitely many identities among $q$-series. For example, in the case of $Q=A_2$, comparing coefficients of $y_{\alpha_1}^2 y_{\alpha_2}^2$ on both sides of Equation \eqref{eqn:dilog.ident.extremes} amounts to the identity
\begin{equation}
	\label{eqn:q-series.identity}
\frac{1}{(1-q)^2(1-q^2)^2} = \frac{1}{(1-q)(1-q^2)} + \frac{q}{(1-q)^3}+\frac{q^4}{(1-q)^2(1-q^2)^2}.
\end{equation}
There is such an identity associated to each $y_{\alpha_1}^{\gamma(1)} y_{\alpha_2}^{\gamma(2)}$ for every choice of $\gamma(1)$ and $\gamma(2)$. In addition to being the Poincar\'e series of the algebra $H^*(B\GL(k,\C))$, the $q$-series $\cP_n$ is also a generating function for counting partitions; i.e. 
\[
\cP_k = \frac{1}{\prod_{j=1}^k (1-q^j)} = \sum_{n\geq 0} \pi(n;k)\,q^n
\]
where $\pi(n;k)$ denotes the number of partitions of the number $n$ using only parts $1,2,\ldots,k$ (or equivalently by taking the transpose of a Ferrer's diagram, with at most $k$ parts). Hence, one point of view on quantum dilogarithm identities like \eqref{eqn:dilog.ident.interpolated} is that they encode infinitely many identities among partition generating functions, e.g.~\eqref{eqn:q-series.identity}, and hence infinitely many partition counting identities. Thus one interesting problem is to mine the quantum dilogarithm identities for new partition counting identities. Some initial progress appeared in \cite{rraway2018}, where a purely combinatorial partition argument produces a new proof of the extremal version (i.e.~Theorem~\ref{thm:reineke.rimanyi}) of the quantum dilogarithm identity \eqref{eqn:dilog.ident.extremes} in the case of $A$-type Dynkin quivers.
\end{remark}


\section{Codimensions of quiver strata from the quantum algebra}
\label{s:count.codim}

The goal of this section is to perform an important calculation in the quantum algebra which will produce the codimensions of the quiver strata $\eta_m$. In particular, we prove the following generalization of \cite[Lemma~5.1]{rr2013}.

\begin{prop}
	\label{prop:compute.codim}
Fix a dimension vector $\gamma \in D_Q$ and let $m\kp \gamma$ be a Kostant series compatible with the admissible Dynkin subquiver partition $\cQ = \{Q^1,\ldots,Q^\ell\}$. Consider the product \[Y_m = y_{\phi_1}^{m_1} y_{\phi_2}^{m_2} \cdots y_{\phi_r}^{m_r} \in \A_Q.\]
We have
\begin{equation}
\label{eqn:positives.to.simples}
Y_m = (-1)^{s_m} \cdot q^{w_m} \cdot y_{e_1}^{\gamma(1)}\cdots y_{e_n}^{\gamma(n)}
\end{equation}
where
\begin{subequations}
\label{eqn:sm.wm.formulas}
\begin{gather}
	\label{eqn:sm.formula}
	s_m = \sum_{u=1}^r m_u\left(\sum_{i\in Q_0} d^i_u -1\right) \\
	\label{eqn:wm.formula}
	w_m = \codim_\C\left(\eta_m;\Rep_\gamma(Q)\right) 
		+ \frac{1}{2}\sum_{i\in Q_0} \gamma(i)^2 
		- \frac{1}{2}\sum_{u=1}^r m_u^2.
\end{gather}
\end{subequations}
\end{prop}

\begin{proof}
Write $\phi_u = \beta^{j(u)}_{k(u)}$ for each $u$. That is, $\phi_u\in \Phi^{j(u)}$ and it is the $k(u)$-th root in the ordering of $\Phi^{j(u)}$ described by \eqref{eqn:pos.roots.order.within}. Also, here and throughout the rest of the paper we freely write $m_u = m^{j(u)}_{k(u)}$ for entries in a Kostant series. Our first goal is to determine the coefficient $\Gamma$ after reordering as below
\begin{equation}
\label{eqn:rearrange.into.subquivers}
Y_m = y_{\beta^{j(1)}_{k(1)}}^{m_1} \cdots y_{\beta^{j(r)}_{k(r)}}^{m_r} 
	= \Gamma\cdot
		\left( y_{\beta^1_1}^{m^1_1} \cdots y_{\beta^1_{r_1}}^{m^1_{r_1}} \right) \cdots 
		\left( y_{\beta^\ell_1}^{m^\ell_1} \cdots y_{\beta^\ell_{r_\ell}}^{m^\ell_{r_\ell}} \right).
\end{equation}
Observe that for $j(u) = j(v)$, by construction we have $k(u)<k(v)$ whenever $u<v$. Thus, within the parentheses on the righthand side of \eqref{eqn:rearrange.into.subquivers}, the products are already ordered consistently with the conditions in \eqref{eqn:pos.roots.order.within}. This means that the only contributions to $\Gamma$ are powers of $q$ which arise from commuting $y_{\phi_u}$ past $y_{\phi_v}$ when $u<v$, but $j(u)>j(v)$. The resulting power of $q$ is $m_u m_v \lambda(\phi_u,\phi_v)$. Thus
\begin{equation}
\label{eqn:Gamma.formula}
\Gamma = q^{\sum m_u m_v \lambda(\phi_u,\phi_v)}
\end{equation}
where the sum is over pairs $(u,v)$ with $u<v$ but $j(u)>j(v)$. Now, whenever $u<v$, $j(u)>j(v)$, and $\lambda(\phi_u,\phi_v) \neq 0$, it means that the subquivers $Q^{j(u)}$ and $Q^{j(v)}$ are connected, necessarily by arrows $a\in Q_1$ having $ha\in Q^{j(u)}_0$ and $ta\in Q^{j(v)}_0$. In fact, since $u<v$ we must have $\lambda(\phi_u,\phi_v)<0$. Then, writing $\phi_u = \sum_i d^i_u \alpha_i = \sum_i d^i_u e_i$ we obtain
\begin{multline*}
m_u m_v \lambda(\phi_u,\phi_v) = m_u m_v \lambda\left( \sum_{i\in Q_0} d^i_u e_i, \sum_{i\in Q_0} d^i_v e_i \right)  \\
	= m_u m_v \lambda\left(d^{ha}_u e_{ha}, d^{ta}_v e_{ta}\right) = -m_u\, d^{ha}_u\, m_v\, d^{ta}_v
\end{multline*}
where we have used \eqref{eqn:lambda.trees} in the last equality. Hence we can re-express $\Gamma$ as follows
\[
\Gamma = q^{-\sum m_u d^{ha}_u m_v d^{ta}_v}
\]
where the sum is over $u<v$, $j(v)<j(u)$, and arrows $a\in Q_1$ such that $d^{ha}_u$ and $d^{ta}_v$ are nonzero. Such arrows must connect distinct subquivers. For fixed $u$, summing over the relevant $v$ gives 
\[
\textstyle \sum_v m_u d^{ha}_u m_v d^{ta}_v = m_u d^{ha}_u \gamma(ta).
\]
Next, we sum over the relevant $u$ to obtain
\[
\Gamma = q^{-\sum_a \gamma(ta)\gamma(ha)}
\]
where the sum is now over arrows $a$ which connect distinct subquivers $Q^{j'}$ and $Q^{j''}$ such that $j'<j''$ and $ta\in Q^{j'}_0$ and $ha \in Q^{j''}_0$.

Next, for each $1\leq j \leq \ell$, we define the product 
\[
Y^j_m =  y_{\beta^j_1}^{m^j_1} \cdots y_{\beta^j_{r_j}}^{m^j_{r_j}};
\]
that is, we have written $Y_m = \Gamma\cdot Y^1_m \cdots Y^\ell_m$. Since each subquiver $Q^j$ is Dynkin, and each $m^j \kp \gamma^j$, we can apply \cite[Lemma~5.1]{rr2013} to obtain that
\[
Y^j_m = (-1)^{\sum_{k=1}^{r_j} m_k^j (\sum_{i\in Q^j_0} d^i_k - 1)} \cdot q^{w^j} \cdot \prod^{\to}_{i\in Q^j_0} y_{e_i}^{\gamma^j(i)}
\]
where the arrow atop the product symbol indicates the multiplication must be done in order (from left to right) of increasing $i$. Further, \cite[Lemma~5.1]{rr2013} gives
\[
w^j = \codim_\C\left(\Omega_{m^j};\Rep_{\gamma^j}(Q^j)\right) + \frac{1}{2}\sum_{i\in Q^j_0} \gamma^j(i)^2 - \frac{1}{2}\sum_{k=1}^{r_j} (m^j_k)^2.
\]
Notice that by construction of our subquivers, we can replace $\gamma^j(i) = \gamma(i)$ for each $i\in Q^j_0$  since each vertex $i\in Q_0$ appears in exactly one subquiver.

Using Proposition \ref{prop:codim.eta} and (again) the fact that each vertex $i\in Q_0$ appears in exactly one subquiver $Q^j_0$, we obtain that
\begin{equation}
	\label{eqn:Ym.halfway}
Y_m = \Gamma \cdot Y^1_m \cdots Y^\ell_m =  \Gamma \cdot (-1)^{s_m} \cdot q^{w_m} \cdot \prod_{1\leq j \leq \ell}^{\to} \left(\prod_{i\in Q_0^j}^{\to} y_{e_i}^{\gamma(i)}\right)
\end{equation}
where
\begin{align*}
s_m &= \sum_{j=1}^\ell \sum_{k=1}^{r_j} m^j_k \left(\sum_{i\in Q_0^j} d^i_k - 1 \right)  &
w_m &= \sum_{j=1}^\ell w^j
\end{align*}
and algebraic simplification and the result of Proposition \ref{prop:codim.eta} yields exactly the stated expressions for $s_m$ and $w_m$.

At this point, we have expressed $Y_m$ in terms of a product of the $y_{e_i}$ variables, but we wish to reorder this product to the form $y_{e_1}^{\gamma(1)} y_{e_2}^{\gamma(2)} \cdots y_{e_n}^{\gamma(n)}$. To do so, we observe that the result of \cite[Lemma~5.1]{rr2013} already ensures that the parenthetical products in \eqref{eqn:Ym.halfway} are properly ordered for fixed $j$. Hence we need to commute $y_{e_{ha}}^{\gamma(ha)}$ to appear before $y_{e_{ta}}^{\gamma(ta)}$ only in the case that $a$ connects distinct subquivers, and $ta\in Q_0^{j'}$ and $ha \in Q_0^{j''}$ where $j'<j''$. We have the commutation relation, true for every $a\in Q_1$, from Equations \eqref{eqn:lambda.trees} and \eqref{eqn:qalg.comm}
\[
y_{e_{ta}}y_{e_{ha}} = q^{\lambda(e_{ta},e_{ha})} y_{e_{ha}}y_{e_{ta}} 
	= q\, y_{e_{ha}}y_{e_{ta}}.
\]
Whence it follows that
\[
\prod_{1\leq j \leq \ell}^{\to} \left(\prod_{i\in Q_0^j}^{\to} y_{e_i}^{\gamma(i)}\right)
	= q^{\sum \gamma(ta)\gamma(ha)} y_{e_1}^{\gamma(1)} \cdots y_{e_n}^{\gamma(n)} = \Gamma^{-1}\cdot y_{e_1}^{\gamma(1)} \cdots y_{e_n}^{\gamma(n)} 
\]
since the sum in the exponent of $q$ is over arrows $a$ as described above. Finally, we see that this implies that
\[
Y_m = (-1)^{s_m} \cdot q^{w_m} \cdot y_{e_1}^{\gamma(1)} \cdots y_{e_n}^{\gamma(n)}
\]
as desired. 
\end{proof}

\begin{remark}
	\label{rem:Gamma.cancels}
The value of the factor $\Gamma$ in \eqref{eqn:rearrange.into.subquivers} depended on the choice of order for the subquivers $Q^1,\ldots,Q^\ell$; that is, we did not assume that $\cQ$ was ordered (we needed the admissibility assumption to ensure no cancellation in the sums $\sum \gamma(ta)\gamma(ha)$ over arrows connecting distinct subquivers from $\cQ$). Since the order of the subquivers $Q^j$ was arbitrary, we should forecast the cancellation of $\Gamma$ in the last displayed equation above since the final step in our proof was the only other consideration which depended on that order. 

In other words, our proof of Proposition \ref{prop:compute.codim} applies even in the general case when the given admissible Dynkin subquiver partition $\cQ$ is not ordered. If it is ordered, then $\Gamma = 1$.
\end{remark}


\section{Reduction to normal forms}
\label{s:reduction.normal.forms}

For the Dynkin subquiver partition $\cQ = \{Q^1,\ldots,Q^\ell\}$ let $m\kp\gamma$ be a compatible Kostant series (the results of this section will apply even if $\cQ$ is not admissible). For each $1\leq j \leq \ell$, we have the normal form $\nu_{m^j}$ corresponding to each subquiver orbit $\Omega_{m^j}$. We define the \emph{normal locus} of the quiver stratum $\eta_m$ to be
\[
\nu_m = \left\{ (x_a)_{a\in Q_1} \in \eta_m : (x_a)_{a\in Q^j_1} = \nu_{m^j} \text{~for all $j$} \right\}.
\]
Observe that, as in the definition of $\eta_m$ from Section \ref{ss:subquiver.parts}, for arrows $a\in Q_1\setminus (\Union_{j=1}^\ell Q^j_1)$ we allow $x_a$ to be arbitrary, and hence we have a natural identification
\begin{equation}
	\label{eqn:nu.m.identification} 
	\nu_m \homeo \Dirsum_{a\in Q_1\setminus (\Union_{j=1}^\ell Q^j_1)} \Hom(V_{ta},V_{ha}) .
\end{equation}
Further, we define the \emph{isotropy subgroup} of the quiver stratum $\eta_m$ to be
\[
\G_{\eta_m} = \{ g\in \G_\gamma : g\cdot \nu_m = \nu_m \}.
\]
 
\begin{prop}
	\label{prop:isotropy.strata}
	There is an isomorphism $\G_{\eta_m} \iso \prod_{j=1}^\ell \G_{\Omega_{m^j}}$. Therefore, up to homotopy we have
	\[
	\G_{\eta_m} \hmtpc \prod_{j=1}^\ell \prod_{k=1}^{r_j} \U(m^j_k).
	\]
\end{prop}

\begin{proof}
The stated isomorphism follows from the definitions, and Proposition \ref{prop:FR} provides the stated homotopy equivalence.
\end{proof}

We now state a lemma which appeared in our joint work with Rim\'anyi \cite[Lemma~8.6]{jarr2018}. The upshot is that we can reduce the computation of the equivariant cohomology algebra of a quiver stratum to its normal locus.

\begin{lem}
	\label{lem:reduce.normal.form}
Let the group $G$ act on the space $X$. Suppose that $A \subseteq X$ is a subspace with isotropy subgroup $G_A = \{g\in G : g\cdot A = A\}$. Assume that
	\begin{itemize}
	\item every $G$-orbit in $X$ intersects $A$;
	\item if $g\in G$ such that there exists $a\in A$ with $g\cdot a \in A$, then $g \in G_A$.
	\end{itemize}
Then we have $H^*_G(X) \iso H^*_{G_A}(A)$. \qed
\end{lem}

We next obtain a result which we will implicitly need in Section \ref{s:KSS}.

\begin{prop}
	\label{prop:normal.form.context}
	There is an isomorphism $H^*_{\G_\gamma}(\eta_m) \iso H^*_{\G_{\eta_m}}(\nu_m)$. Moreover, since $\nu_m$ is a $\G_{\eta_m}$-equivariantly contractible vector space, we also have an isomorphism $H^*_{\G_{\eta_m}}(\nu_m) \iso H^*(B\G_{\eta_m})$.
\end{prop}

\begin{proof}
The first isomorphism is Lemma \ref{lem:reduce.normal.form} with $X = \eta_m$, $G = \G_\gamma$, $A = \nu_m$, and $G_A = \G_{\eta_m}$. The second isomorphism follows because of the identification, see \eqref{eqn:nu.m.identification}, of $\nu_m$ with a $\G_{\eta_m}$-equivariantly contractible vector space.
\end{proof}


\section{Kazarian spectral sequence}
\label{s:KSS}

Let the Lie group $G$ act on the real manifold $X$. Furthermore, suppose that $X$ admits a stratification $X = \theta_1 \union \theta_2 \union \cdots \union \theta_u$ into $G$-invariant submanifolds. Observe that only finitely many strata are permitted. Define the $G$-invariant spaces
\[
F_i = \mathop{\Union_{1\leq j \leq u}}_{\codim_\R(\theta_j;X)\leq i} \theta_j
\]
to obtain the following topological filtration of $X$:
\[
F_0 \subseteq F_1 \subseteq \cdots \subseteq F_{\dim_\R(X)} = X.
\]
Taking the \emph{Borel construction} $B_G F_i = EG \times_G F_i$ of each stratum yields a topological filtration of $B_G X$
\[
B_G F_0 \subseteq B_G F_1 \subseteq \cdots \subseteq B_G F_{\dim_\R(X)} = B_G X.
\]
We comment that we have already used $BG$ to denote $B_G(\text{point})$; we continue to do so in the rest of the paper.

Following the terminology of \cite{jarr2018,rr2013}, we call the cohomological spectral sequence associated to this last filtration the \emph{Kazarian spectral sequence of the action} and/or \emph{of the stratification}. The name is given in honor of Kazarian's contribution in the analogous context of singularity theory and Thom polynomials \cite{mk1997}, but the technique appears also in various other (even earlier) works, see e.g.~\cite{marb1983}.

We will apply the Kazarian spectral sequence in the context that $X = \Rep_\gamma(Q)$, $G = \G_\gamma$, and the $\theta_i$ are the quiver strata $\eta_m$. The needed results are summarized by the next theorem.

\begin{thm}
\label{thm:KSS.context}
In the above context, the Kazarian spectral sequence converges to $H^*(B\G_\gamma)$, degenerates at the $E_1$ page, and we have isomorphisms
\[
\pushQED{\qed}
E^{c,j}_1 \iso \mathop{\Dirsum_{m\kp\gamma}}_{ \codim_\R(\eta_m;\Rep_\gamma(Q)) = c } H^j\left(B\G_{\eta_m}\right).
\qedhere
\popQED
\]
\end{thm}

\begin{remark}[on the details of Theorem \ref{thm:KSS.context}]
The proof of the theorem above follows as a special case of \cite[Theorem~9.1]{jarr2018}. In that work, the more general setting of equivariant rapid decay cohomology is used. However, we still remark on a few particulars that allow the application of the Kazarian spectral sequence in this context, \textit{cf.}~\cite[Section~9.2]{jarr2018}. 

The convergence claim follows from the fact that $\Rep_\gamma(Q)$ is equivariantly contractible, and so $H^*_{\G_\gamma}(\Rep_\gamma(Q)) \iso H^*(B\G_\gamma)$. The sequence degenerates at page $E_1$ since $\Rep_\gamma(Q)$ and each of the strata $\eta_m$ are actually complex manifolds and thus every nonzero contribution on the first page must have the form $E^{c,j}_1$ where both $c$ and $j$ are even.

The description of the $E_1$ page comes from its definition in terms of relative cohomologies, after applying excision and the Thom isomorphism. Details from the original work \cite{mk1997} are cited in \cite{rr2013} in the context where the strata are actually orbits. However, in Theorem \ref{thm:KSS.context} we have a slightly more robust version where each stratum consists of a union of orbits. The details in this case appear in the author's aforementioned joint work with Rim\'anyi \cite[Theorem~9.1]{jarr2018}. Comparing to that statement, the direct summands should be $H^j_{\G_\gamma}(\eta_m)$, but Proposition \ref{prop:normal.form.context} allows us to replace $H^j_{\G_\gamma}(\eta_m)$ with $H^j(B\G_{\eta_m})$ in the direct summands above.
\end{remark}

The main conclusion we need from Theorem \ref{thm:KSS.context} is the following $q$-series identity.

\begin{cor}
\label{cor:KSS.Betti.identity}
For every dimension vector $\gamma$, we have the following identity for Betti numbers
\begin{equation}
	\label{eqn:kss.betti.ident}
	\cP_{\gamma(1)}\cdots\cP_{\gamma(n)} 
	= \sum_{m\kp\gamma} 
		q^{\codim_\C(\eta_m;\Rep_\gamma(Q))} \,\cP_{m_1}\cdots\cP_{m_r}.
\end{equation}
We remark that $r$ is a function of $m$ in the summation above, and that $m$ ranges over Kostant series compatible with a fixed Dynkin subquiver partition $\cQ$.
\end{cor}

\begin{proof}
The homeomorphism $B\G_\gamma \homeo \prod_{i\in Q_0} B\GL(\C^{\gamma(i)})$ and the K\"unneth formula imply that $H^*(B\G_\gamma) \iso \Tensor_{i\in Q_0} H^*(B\GL(\C^{\gamma(i)}))$. Thus, the Poincar\'e series of the algebra $H^*(B\G_\gamma)$ is exactly the lefthand side of \eqref{eqn:kss.betti.ident}.

On the other hand, Proposition \ref{prop:isotropy.strata} and the K\"unneth formula similarly imply that \[\cP[H^*(B\G_{\eta_m})] = \prod_{u=1}^r \cP_{m_u}.\] Finally, since the Kazarian spectral sequence converges to $H^*(B{\G_\gamma})$ and degenerates at the $E_1$ page, Theorem \ref{thm:KSS.context} implies the required identity.
\end{proof}

\begin{example}
Observe that the $q$-series identity \eqref{eqn:q-series.identity} is a special case of Corollary \ref{cor:KSS.Betti.identity} with $Q=A_2$, $\cQ = \{1\leftarrow 2\}$, $\gamma=(2,2)$, and so $\ell = 1$ and $r=|\Phi_{A_2}| = 3$. 
\end{example}

Further, we remark again that the results of this section apply even when the given Dynkin subquiver partition is not admissible. However, although we obtain $q$-series identities for each dimension vector $\gamma$ (this is Corollary \ref{cor:KSS.Betti.identity}) they cannot be simultaneously organized into a quantum dilogarithm factorization of $\E_Q$ as in Theorem \ref{thm:main} unless a total order on the roots $\Phi(Q,\cQ)$ exists; that is, unless our chosen Dynkin subquiver partition is admissible.


\section{Proof of the main theorem}
\label{s:main.thm.pf}

Our goal is to now prove Theorem \ref{thm:main}, and so we restate its conclusion. For any acyclic quiver $Q$, admissible Dynkin subquiver partition $\cQ = \{Q^1,\ldots,Q^\ell\}\dsqp Q$, and total ordering
\[
\phi_1 \prec \cdots \prec \phi_r
\]
on the roots $\Phi(Q,\cQ)$ as described in \eqref{eqn:order.rules.admissible}, we obtain the following factorization of the combinatorial DT invariant for $Q$
\begin{equation}
	\label{eqn:main.restate}
	\E_Q = \E(y_{\phi_1}) \cdots \E(y_{\phi_r}).
\end{equation}
We will proceed by comparing, for a fixed dimension vector $\gamma \in D_Q$, the terms involving \[\yy^\gamma := y_{e_1}^{\gamma(1)}\cdots y_{e_n}^{\gamma(n)}\] on the lefthand and righthand sides of \eqref{eqn:main.restate}. On the left, we apply our knowledge of the trivial factorization \eqref{eqn:trivial.factorization} and the formula \eqref{eqn:E.defn.poincare} to see that the coefficient of $\yy^\gamma$ is
\begin{equation}
	\label{eqn:lhs.coeff}
	(-1)^{\sum_{i\in Q_0} \gamma(i)}\,q^{\frac{1}{2} \sum_{i\in Q_0} \gamma(i)^2}
		\cP_{\gamma(1)} \cdots \cP_{\gamma(n)}.
\end{equation}
On the right, we get contributions to the $\yy^\gamma$ term only from expressions involving $Y_m = y_{\phi_1}^{m_1} \cdots y_{\phi_r}^{m_r}$ for which $m\kp \gamma$ is a compatible Kostant series. Therefore, using the notation of Proposition \ref{prop:compute.codim}, we see that the $\yy^\gamma$ term is equal to the sum
\[
\sum_{m\kp \gamma} (-1)^{\sum_{u=1}^r m_u} \, q^{\frac{1}{2}\sum_{u=1}^r m_u^2} \, \cP_{m_1}\cdots\cP_{m_r} \, Y_m 
\] 
which, by Equation \eqref{eqn:positives.to.simples} from Proposition \ref{prop:compute.codim}, is further equal to
\[
\left( \sum_{m\kp \gamma} (-1)^{s_m + \sum_{u=1}^r m_u} \, q^{w_m + \frac{1}{2}\sum_{u=1}^r m_u^2} \, \cP_{m_1}\cdots\cP_{m_r} \right) \yy^\gamma.
\]
Plugging in the formulas \eqref{eqn:sm.formula} for $s_m$ and \eqref{eqn:wm.formula} for $w_m$ from Proposition \ref{prop:compute.codim}, and using that $\gamma(i) = \sum_{u=1}^r m_u d^i_u$ for all $i\in Q_0$, we see that the resulting coefficient of $\yy^\gamma$ is
\begin{multline}
	\label{eqn:rhs.coeff}
	\sum_{m\kp\gamma} 
		(-1)^{\sum_{i\in Q_0}\gamma(i)} \,
		q^{\codim_\C(\eta_m;\Rep_\gamma(Q)) + \frac{1}{2}\sum_{i\in Q_0} \gamma(i)^2} \,
		\cP_{m_1}\cdots\cP_{m_r}	\\
	= (-1)^{\sum_{i\in Q_0}\gamma(i)} \, q^{\frac{1}{2}\sum_{i\in Q_0} \gamma(i)^2} \,
		\sum_{m\kp\gamma} q^{\codim_\C(\eta_m;\Rep_\gamma(Q))} \,\cP_{m_1}\cdots\cP_{m_r}.
\end{multline}
Finally, we conclude that \eqref{eqn:lhs.coeff} and \eqref{eqn:rhs.coeff} are equal because of the Kazarian spectral sequence identity \eqref{eqn:kss.betti.ident} from Corollary \ref{cor:KSS.Betti.identity}. \qed

\section{Applicable directions}
\label{s:FD}
We comment now on several other possible generalizations and applications of the present work. 

\subsection{Interpolations in the non-acyclic case}
	\label{ss:interp.non-acyclic}
Keller describes the existence of maximal green sequences for many other classes of quivers (not just acyclic) \cite[Section~6]{bk2011}. However, in the non-acyclic case, one must invoke the theory of quivers with potential and compute Betti numbers in the so-called \emph{rapid decay} equivariant cohomology of quiver strata, see e.g.~\cite[Section~4.7]{mkys2011}. This is difficult in general, but the author and Rim\'anyi accomplished this in the case of square products of $A$-type Dynkin quivers \cite{jarr2018}. In that work, two natural stratifications of the representation space are defined, but it would be interesting to describe more general strata which allow interpolation between the quantum dilogarithm identities in that work.

\subsection{Combinatorial interpretations}
Rim\'anyi--Weigandt--Yong produced a tableaux counting proof for the extremal factorizations in the Dynkin type $A$ case \cite{rraway2018}. In the special case of $A_2$, their argument reduces to the famous Durfee's square identity for counting partitions. It would be interesting to extend their methods to a purely combinatorial proof of the interpolating factorizations of Theorem \ref{thm:main}. On the level of $q$-series, this would lead to new partition counting identities.

\subsection{Counting maximal green sequences}
The problem of counting the number and length of possible maximal green sequences for a quiver has received considerable attention. Several papers have addressed the bounds on the length of such sequences \cite{tbgdmp2014,shki2019,tbshkigt2017,agtmks2018} on general classes of quivers (and mutation types). Our present work can be viewed in one sense as a confirmation of these results on the minimal length of maximal green sequences, and on the maximal length in the Dynkin case. Further, it would be interesting to investigate connections to the ``No Gap Conjecture'', which states that the possible lengths of maximal green sequences (and hence for factorizations of the combinatorial DT invariant) form an interval of integers, see \cite[Conjecture~2.2]{tbgdmp2014}. In the work of Hermes--Igusa \cite{shki2019}, the conjecture is proven for acyclic quivers of tame type. The conjecture is also known to hold in several other instances, including for some non-acyclic quivers, but is still open in general for non-acyclic and wild cases, see e.g.~\cite{agtm2019} and the references therein.


\bibliographystyle{plainurl}
\bibliography{jmabib}
\end{document}